\documentclass[12pt]{amsart}
\usepackage{amsmath,amsfonts,amssymb} \usepackage{color}
\usepackage[all,cmtip]{xy}
\usepackage{graphicx}
 \usepackage{youngtab}

 \newcommand{\bR}{\mathbb{R}}

 \newcommand{\bP}{\mathbb{P}} \newcommand{\cO}{\mathcal{O}}

 \newcommand{\bZ}{\mathbb{Z}}
 \newcommand{\bC}{\mathbb{C}}

 \DeclareMathOperator{\Ell}{Ell}
 \DeclareMathOperator{\res}{res}

\newcommand{\be}{\begin{equation}}
\newcommand{\ee}{\end{equation}}
\newcommand{\bea}{\begin{eqnarray}}
\newcommand{\eea}{\end{eqnarray}}
\newcommand{\ben}{\begin{eqnarray*}}
\newcommand{\een}{\end{eqnarray*}}

\newcommand{\half}{\frac{1}{2}}

\newtheorem{cor}{Corollary}[section]

 \newtheorem{thm}[cor]{Theorem}
\theoremstyle{remark}

\definecolor{A}{rgb}{.75,1,.75}

\definecolor{green}{rgb}{0,1,0}
\definecolor{yellow}{rgb}{1,1,0}
\definecolor{orange}{rgb}{1,.7,0}
\definecolor{red}{rgb}{1,0,0}
\definecolor{white}{rgb}{1,1,1}

 \makeindex
\begin{document}
\title
{On Equivariant Elliptic Genera of Toric Calabi-Yau 3-folds}
%\author{ }
%\thanks{ }

\author{Jian Zhou}
\address{Department of Mathematical Sciences\\Tsinghua University\\Beijng, 100084, China}
\email{jzhou@math.tsinghua.edu.cn}

\begin{abstract}
We show that equivariant elliptic genera of toric Calabi-Yau 3-folds
are generalized weak Jacobi forms.
We also introduce a notion of averaged  equivariant elliptic genera
of toric Calabi-Yau 3-folds, and show that they  are ordinary
weak Jacobi forms given by an explicit formula predicted by
 Eguchi and Sugawara.
\end{abstract}

\dedicatory{Dedicated to the memory of Professor S.-S. Chern on the occasion of 
his 104th birthday.}

\maketitle

\section{Introduction}

Geometrically,
by the elliptic genus of a compact complex manifold $Y$ of dimension $n$ we
mean the Euler characteristic:
\be
Z_Y(\tau, z) = \chi(Y, \Ell(T^*Y; \tau, z)),
\ee
where for a vector bundle $E$ on $Y$,
$\Ell(V; \tau, z) \in y^{-n/2} K(Y)[[q, y,y^{-1}]]$ is defined by:
\be
\begin{split}
& \Ell(V; \tau, z) \\
=& q^{-n/2} \bigotimes_{m \geq 1}
(\Lambda_{-yq^{m-1}}(V) \otimes \Lambda_{-y^{-1} q^m}(V^*)
\otimes S_{q^m}(V) \otimes S_{q^m}(V^*)).
\end{split}
\ee
Here $q= e^{2\pi i \tau}$, $y = e^{2\pi iz}$.
When $Y$ is Calabi-Yau,
$Z_Y(\tau, z)$ is a weak Jacobi form of weight $0$ and index $n/2$ (see e.g.
\cite{Kawai-Yamada-Yang, Borisov-Libgober}).
Since there is a unique (up to constant) weak Jacobi form of index $3/2$ given by
\be
\frac{\theta_1(\tau, 2z)}{\theta_1(\tau, z)},
\ee
when $n =3$,
one has
\be \label{eqn:ES}
Z_Y(\tau, z) = \frac{\chi(Y)}{2} \frac{\theta_1(\tau, 2z)}{\theta_1(\tau,z)}.
\ee

Physically,
elliptic genus of a compact Calabi-Yau manifold can be defined as a supersymmtric index
of an $N=2$ superconformal field theory associated to it
(see e.g. \cite{Witten, Eguchi-Ooguri-Taormina-Yang, Kawai-Yamada-Yang}),
thus it can then be expressed in terms of some $N=2$ characters.
Such expressions can often be shown to match with
the geometric definition by computing residues of some elliptic functions
(see \cite{Ma-Zhou, Guo-Zhou1, Guo-Zhou2}).

For a noncompact Calabi-Yau manifold,
there are still physical definition and computations of
the elliptic genera.
Eguchi and Sugawara \cite{Eguchi-Sugawara} conjectured that
\eqref{eqn:ES} also holds for noncompact Calabi-Yau 3-folds.
Unfortunately elliptic genera of noncompact complex manifolds have not been defined
geometrically in general.
To make sense of this conjecture,
one has to first make a suitable definition.

In this paper we will focus on toric Calabi-Yau 3-fold and use the torus
action to  define equivariant elliptic genera.
We will not use the full 3-torus but instead
restrict to a 2-torus that preserves the holomorphic volume form.
We will show that the equivariant elliptic genera
of toric Calabi-Yau 3-folds are generalized weak Jacobi forms
in the sense of \S \ref{sec:Jacobi}.
We will also introduce a notion of balanced toric Calabi-Yau 3-fold.
A typical example is the resolved conifold $\cO(-1) \oplus \cO(-1) \to \bP^1$.
We prove \eqref{eqn:ES} for balanced toric Calabi-Yau 3-folds
by establishing it first for the resolved conifold.
For general toric CY 3-folds,
we introduce a notion of averaged equivariant elliptic genera
that are suitable for the purpose.

It is very interesting to compare the compact case and the noncompact case.
In both cases some kind of localization formula is used.
In the former it is Cauchy's residue formula,
in the latter it is Atiyah-Bott's Lefschetz formula \cite{Atiyah-Bott}.

The rest of the paper is arranged as follows.
In \S 2 we show that the equivariant elliptic genera of
toric Calabi-Yau 3-folds are generalized weak Jacobi forms.
In \S 3 we define balanced toric Calabi-Yau 3-folds and show that
they satisfy \eqref{eqn:ES}.
We also define averaged equivariant elliptic genera
and show that they are given by \eqref{eqn:ES}.
To conclude this paper, we propose a definition of elliptic genera of general noncompact
complex manifolds in \S 4.

\vspace{.1in}

{\em Acknoledgements}.
This research is partially supported by NSFC grant 11171174.
The author thanks Professor Eguchi for bringing his attention to his joint work
with Professor Sugawara about ten years ago.

\section{Equivariant Elliptic Genera of Toric Calabi-Yau 3-Folds}

In this section we first briefly recall the definition of equivariant elliptic genera
of toric Calabi-Yau 3-folds.
(This is based on Atiyah-Bott-Lefschetz formula \cite{Atiyah-Bott}.
See \cite{LLZ} for more examples of computations of equivariant indices
using this formula. )
Next we show that they are generalized weak Jacobi forms.

\subsection{Toric Calabi-Yau 3-folds}

Such spaces can be described by planar diagrams $\Gamma$ satisfying the following conditions.
Each vertex either has three incident edges (it will then be called a trivalent vertex)
or has only one incident edge (it will then be called a univalent vertex).
At each trivalent vertex $v$ of $\Gamma$,
the three incident outgoing edges are in the directions of
three vectors  $w_1^v=(a_1^v, b_1^v)$, $w_2^v= (a_2^v, b_2^v)$
and $w^v_3 = (a_3^v, b_3^v)$ in $\bZ^2$ respectively.
These vectors will be called weight vectors at $v$ satisfies the following conditions:
\begin{itemize}
\item[(a)] (Balancing at trivalent vertices) For all $v \in V(\Gamma)$,
where $V(\Gamma)$ denotes the set of trivalent vertices of $\Gamma$,
$w_1^v + w_2^v + w^v_3 = 0$.
\item[(b)] (Balancing along internal edges) For each internal edge $e$ of $\Gamma$
joining two trivalent vertices $v_1$ and $v_2$,
we have $w^{v_1}_{v_1, e} + w^{v_2}_{v_2, e} = 0$,
where if $e$ is an edge incident at a trivalent vertex $v$,
$w^v_{v, e}$ denotes the weight vector at $v$ along the edge $e$.
\end{itemize}
We understand the weight vectors as the weights of the torus action
on the cotangent spaces of the fixed points.
The following are two examples:
They are the total spaces of $\cO(-1) \oplus \cO(-1) \to \bP^1$
and $\cO(-3) \to \bP^2$ respectively.
$$
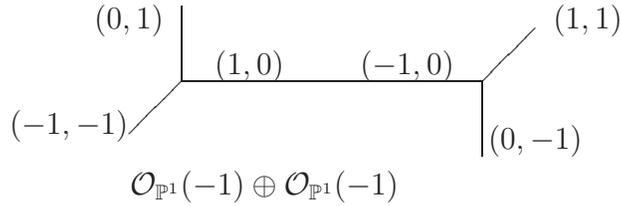
\xy
(0,10); (0,0), **@{-}; (-7, -7), **@{-};
(0,0); (40,0), **@{-}; (47,7), **@{-}; (40,0); (40, -10), **@{-};
(-7, 8)*+{(0,1)};
(-15, -6)*+{(-1,-1)}; (9,2)*+{(1,0)}; (30,2)*+{(-1,0)};
(11,-14)*+{\cO_{\bP^1}(-1)\oplus \cO_{\bP^1}(-1)};
(47, -8)*+{(0,-1)};  (54, 8)*+{(1,1)};
\endxy
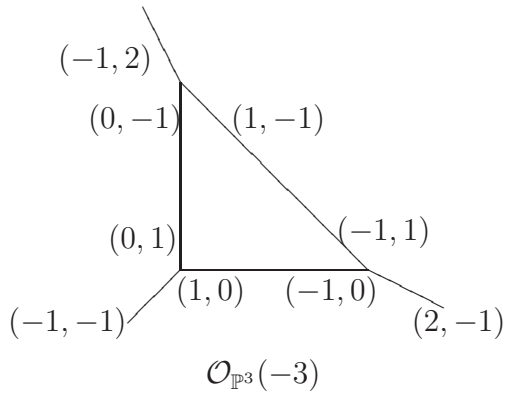
$$

$$
\xy
(-7,-7); (0,0), **@{-}; (25, 0), **@{-}; (0, 25), **@{-};
(-5, 35), **@{-};
(0,25); (0,0), **@{-}; (25, 0), **@{-}; (35, -5), **@{-};
(-15, -7)*+{(-1,-1)}; (4,-3)*+{(1,0)}; (20,-3)*+{(-1,0)}; (37, -7)*+{(2,-1)};
(-5,4)*+{(0,1)}; (27,5)*+{(-1,1)};
(-6,20)*+{(0,-1)}; (13,20)*+{(1,-1)};
(-10,28)*+{(-1,2)};
(11,-14)*+{\cO_{\bP^3}(-3)};
\endxy
$$

\subsection{Equivariant elliptic genera of toric Calabi-Yau 3-folds}

We will identify the vector $w_j^v = (a^v_j, b^v_j)$ with $a^v_j t_1 + b^v_j t_2$,
also denoted by $w_j^v$.
The equivariant elliptic genus of a toric Calabi-Yau 3-fold $Y$ associated to
a toric diagram $\Gamma$ is given by
\be
Z_Y(\tau, z; t_1, t_2)
= \sum_{v \in V(\Gamma)}
\prod_{j=1}^3 \frac{\theta_1(\tau, z+w^v_j)}{\theta_1(\tau, w^v_j)},
\ee
where $\theta_1$ is the theta function defined by:
\be
\begin{split}
\theta_1(\tau, z) = & i
\sum_{n=-\infty}^\infty
(-1)^n q^{(n-1/2)^2/2}
y^{n-1/2} \\
%%% = & 2 \sin( \pi z)q^{1/8} \prod_{m=1}^\infty (1 - q^m)(1 - y q^m)(1 - y^{-1}q^m) \\
= & i q^{1/8} y^{-1/2} \prod_{m=1}^\infty (1 - q^m)(1 - y q^{m-1})(1 - y^{-1}q^m).
\end{split}
\ee
For example,
\be \label{eqn:Z-conifold}
\begin{split}
& Z_{\cO_{\bP^1}(-1)^{\oplus 2}}(\tau, z; t_1, t_2) \\
= & \frac{\theta_1(\tau, z+t_1)}{\theta_1(\tau, t_1)}
\frac{\theta_1(\tau, z+t_2)}{\theta_1(\tau, t_2)} \cdot
\frac{\theta_1(\tau, z-t_1-t_2)}{\theta_1(\tau, -t_1-t_2)} \\
+ & \frac{\theta_1(\tau, z-t_1)}{\theta_1(\tau, -t_1)} \cdot
\frac{\theta_1(\tau, z-t_2)}{\theta_1(\tau, -t_2)} \cdot
\frac{\theta_1(\tau, z+t_1+t_2)}{\theta_1(\tau, t_1+t_2)}.
\end{split}
\ee

\be
\begin{split}
& Z_{\cO_{\bP^2}(-3)}(\tau, z; t_1, t_2) \\
= & \frac{\theta_1(\tau, z+t_1)}{\theta_1(\tau, t_1)}
\frac{\theta_1(\tau, z+t_2)}{\theta_1(\tau, t_2)} \cdot
\frac{\theta_1(\tau, z-t_1-t_2)}{\theta_1(\tau, -t_1-t_2)} \\
+ & \frac{\theta_1(\tau, z-t_1)}{\theta_1(\tau, -t_1)} \cdot
\frac{\theta_1(\tau, z-t_1 + t_2)}{\theta_1(\tau, -t_1+ t_2)} \cdot
\frac{\theta_1(\tau, z+2t_1-t_2)}{\theta_1(\tau, 2t_1-t_2)} \\
+ & \frac{\theta_1(\tau, z-t_2)}{\theta_1(\tau, -t_2)} \cdot
\frac{\theta_1(\tau, z-t_2 + t_1)}{\theta_1(\tau, -t_2+ t_1)} \cdot
\frac{\theta_1(\tau, z+2t_2-t_1)}{\theta_1(\tau, 2t_2-t_1)}.
\end{split}
\ee

\subsection{Modular transformation properties of equivariant elliptic genera of toric Calabi-Yau 3-folds}

Recall that the theta-function $\theta_1$ has the following well-known
modular transformation properties:
\begin{align}
\theta_1(\tau, z+1) & = - \theta_1(\tau, z), &
\theta_1(\tau, z+\tau) & = - e^{-2\pi i z - \pi i \tau} \theta_1(\tau, z), \\
\theta_1(\tau+1, z) & = e^{\pi i/4} \theta_1(\tau, z), &
\theta_1(- \frac{1}{\tau}, \frac{z}{\tau})
& = - i \sqrt{\frac{\tau}{i}} e^{\frac{\pi i z^2}{\tau}} \theta_1(\tau, z).
\end{align}
From these it is straightforward to deduce the following transformation formulas:
\bea
&& Z_Y (\tau, z+1; t_1, t_2) = (-1)^3 \cdot Z_Y(\tau, z; t_1, t_2), \\
&& Z_Y(\tau, z+\tau; t_1, t_2) = (-e^{-2\pi iz - \pi i \tau})^3 \cdot Z_Y(\tau, z; t_1, t_2), \\
&& Z_Y(\tau, z; t_1+1, t_2) =  Z_Y(\tau, z, \tau; t_1, t_2), \\
&& Z_Y(\tau, z; t_1+\tau, t_2) =  Z_Y(\tau, z; t_1, t_2), \\
&& Z_Y(\tau, z; t_1, t_2+1) =  Z_Y(\tau, z; t_1, t_2), \\
&& Z_Y(\tau, z; t_1, t_2+\tau) =    Z_Y(\tau, z; t_1, t_2), \\
&& Z_Y(\tau+1, z; t_1, t_2) = Z_Y(\tau, z; t_1, t_2), \\
&& Z_Y(-\frac{1}{\tau}, \frac{z}{\tau}; \frac{t_1}{\tau}, \frac{t_2}{\tau})
= e^{3 \cdot \frac{\pi i z^2}{\tau}} Z_Y(\tau, z; t_1, t_2).
\eea

\subsection{Generalized weak Jacobi forms} \label{sec:Jacobi}

We make the following definition.
Suppose that $\phi: H \times \bC^{n+1} \to \bC$
is a holomorphic function,
where $H$ is the upper half-plane,
such that
\bea
&& \phi(\tau, z+1; t_1, \dots, t_n) = (-1)^{2r} \cdot \phi(\tau, z; t_1, \dots, t_n), \\
&& \phi(\tau, z+\tau; t_1, \dots, t_n)
= (-e^{-2\pi iz - \pi \tau})^{2r} \cdot \phi(\tau, z; t_1, \dots, t_n), \\
&& \phi(\tau, z; t_1, \dots, t_j+1, \dots, t_n)
=  \phi(\tau, z; t_1, \dots, t_n), j=1, \dots, n, \\
&& \phi(\tau, z; t_1, \dots, t_j+\tau, \dots, t_n )
=  \phi(\tau, z; t_1, \dots, t_n), j=1, \dots, n, \\
&& \phi(\tau+1, z; t_1, \dots, t_n) = \phi(\tau, z; t_1, \dots, t_n), \\
&& \phi(-\frac{1}{\tau}, \frac{z}{\tau}; \frac{t_1}{\tau}, \dots, \frac{t_n}{\tau})
= e^{2r \cdot \frac{\pi z^2}{\tau}} \phi(\tau, z; t_1, \dots, t_n).
\eea
Furthermore,
$\phi$ is assumed to have a Fourier expansion with nonnegative powers of $q$.
Then we say $\phi$ is a generalized weak Jacobi form of weight $0$ and index $r$,
with extra variables $t_1, \dots, t_n$.
With this definition,
we then have

\begin{thm}
The equivariant elliptic genus of a toric Calabi-Yau 3-fold
is a generalized weak Jacobi form of weight $0$ and index $3/2$,
with extra variables $t_1$ and $t_2$.
\end{thm}

\section{Averaged Equivariant Elliptic Genera of Toric Calabi-Yau 3-Folds}

In this section we present some examples of toric Calabi-Yau 3-folds
called balanced toric CY 3-filds
whose equivariant elliptic genera are independent of $t_1$ and $t_2$
and given by the prediction of Eguchi-Sugawara.
For general toric CY 3-fold we introduce a notion of
averaged equivariant elliptic genera that have the same property.

\subsection{The equivariant elliptic genus of the resolved conifold}

Our main result is:

\begin{thm} \label{thm:Conifold}
The equivariant elliptic genus of the resolved conifold
is independent of $t_1$ and $t_2$,
and it is given by:
\be
Z_{\cO_{\bP^1}(-1)^{\oplus 2}}(\tau, z)
= \frac{\theta_1(\tau, 2z)}{\theta_1(\tau, z)}.
\ee
\end{thm}

This is in agreement with the prediction by Eguchi-Sugawara \cite[(3.48)]{Eguchi-Sugawara}:
\be
Z(\tau, z) = \frac{\chi}{2} \frac{\theta_1(\tau, 2z)}{\theta_1(\tau, z)},
\ee
since $\chi(\cO_{\bP^1}(-1)^{\oplus 2}) = 2$.
There is a difference of a factor $1/2$ with (3.39) of the same paper.

\begin{proof}
Recall
\be
\begin{split}
& Z_{\cO_{\bP^1}(-1)^{\oplus 2}}(\tau, z, t_1, t_2) \\
= & y^{-3/2} \cdot A(q, y, u) \cdot A(q, y, v) \cdot A(q, y, (uv)^{-1}) \\
+ & y^{-3/2} \cdot A(q, y, u^{-1}) \cdot A(q, y, v^{-1}) \cdot A(q, y, uv),
\end{split}
\ee
where
\be
A(q,y,u) = \prod_{n=1}^\infty \frac{(1-yuq^{n-1})(1-(yu)^{-1}q^n)}{(1-uq^{n-1})(1-u^{-1}q^n)},
\ee
$q=e^{2\pi i \tau}$, $y = e^{2\pi i z}$, $u = e^{2\pi i t_1}$ and $v = e^{2\pi i t_2}$.
As meromorphic functions in $u$,
$A(q, y, u)$ and $A(q, y, u^{-1})$ have first order poles at $q^n$, $n \in \bZ$,
$A(q, y, u)$ and $A(q, y, u^{-1})$ have first order poles at $v^{-1} q^n$, $n \in \bZ$.
It follows that as a meromorphic function in $u$,
$Z_{\cO_{\bP^1}(-1)^{\oplus 2}}(q,y, u, v)$ has only simple poles at
$t_1=m\tau + n$ and $t_1= -t_2+ m\tau + n$ ($m, n \in \bZ$).
By some straightforward calculations,
one gets for all $m \in \bZ$,
\ben
&& \res_{u=q^m} A(q,z, u) du = - q^m y^{-m} \cdot B(q, y), \\
&& \res_{u=q^m} A(q, y, u^{-1}) du = q^m y^{m} \cdot B(q, y), \\
&& A(q,y,uv)|_{u = q^m} = y^{-m} A(q, y, v), \\
&& A(q,y,(uv)^{-1})|_{u = q^m} = y^m A(q, y, v^{-1}),
\een
\ben
A(q, y, u)|_{u =v^{-1}q^m} & = & y^{-m} A(q, y, v^{-1}), \\
A(q,y,u^{-1})|_{u=v^{-1} q^m}   & = & y^m A(q, y, v), \\
\res_{u=v^{-1}q^m}A(q, y, uv) du & = & - v^{-1}q^m y^{-m} B(q, y) \\
\res_{u=v^{-1}q^m}A(q, y, (uv)^{-1})  du& = & v^{-1}q^m y^m B(q, y),
\een
where
\be
B(q,y) = \prod_{n=1}^\infty \frac{(1-yq^{n-1})(1-y^{-1}q^n)}{(1-q^n)^2}.
\ee
From these one can easily deduce that
\bea
&& \res_{t_1= m\tau + n} (Z_{\cO_{\bP^1}(-1)^{\oplus 2}}(\tau, z, t_1, t_2) dt_1)= 0, \\
&& \res_{t_1= -t_2 + m\tau + n} (Z_{\cO_{\bP^1}(-1)^{\oplus 2}}(\tau, z, t_1, t_2) dt_2)= 0,
\eea
for all $m, n \in \bZ$.
It follows that $Z_{\cO_{\bP^1}(-1)^{\oplus 2}}(\tau, z, t_1, t_2)$ is holomorphic
in $t_1$.
Since it is double periodic with periods $1$ and $\tau$,
it is independent of $t_1$.
So is it in $t_2$ by the obvious symmetry between $t_1$ and $t_2$.
Therefore,
$Z_{\cO_{\bP^1}(-1)^{\oplus 2}}(\tau, z, t_1, t_2)$
is a weak Jacobi form of weight $0$ and index $3/2$.
The proof is completed by the fact that \cite{Borisov-Libgober} the space of such forms is one-dimensional and is spanned
by $\frac{\theta_1(\tau, 2z)}{\theta_1(\tau, z)}$.
\end{proof}

By this Theorem we obtain the following  identity for theta function:
\be \label{eqn:Theta-Identity}
\begin{split}
& \frac{\theta_1(\tau, z+t_1)}{\theta_1(\tau, t_1)}
\frac{\theta_1(\tau, z+t_2)}{\theta_1(\tau, t_2)} \cdot
\frac{\theta_1(\tau, z-t_1-t_2)}{\theta_1(\tau, -t_1-t_2)} \\
+ & \frac{\theta_1(\tau, z-t_1)}{\theta_1(\tau, -t_1)} \cdot
\frac{\theta_1(\tau, z-t_2)}{\theta_1(\tau, -t_2)} \cdot
\frac{\theta_1(\tau, z+t_1+t_2)}{\theta_1(\tau, t_1+t_2)} \\
= & \frac{\theta_1(\tau, 2z)}{\theta_1(\tau, z)}.
\end{split}
\ee

\subsection{Balanced toric Calabi-Yau 3-folds}

Suppose that $Y$ is a toric Calabi-Yau 3-fold.
If for any vertex $v$
of its toric graph $\Gamma$,
there is another vertex $v'$ of $\Gamma$ such that for suitable ordering of
the weights,
\be
w_j^v + w_j^{v'} = 0, \;\;\; j=1, 2,3,
\ee
then we say $Y$ is a balanced toric 3-fold.
We will refer to $v'$ as the balancing vertex of $v$.

There are many examples of balanced toric Claabi-Yau 3-folds.
For example,
the canonical line bundles of $\bP^1 \times \bP^1$,
$\bP^1 \times \bP^1$ blown up at two points
$([1:0], [1:0])$ and $([0:1], [0:1])$,
and $\bP^1 \times \bP^1$ blown up at four points
$([1:0], [1:0])$, $([1:0], [0:1])$,$([0:1], [1:0])$, and $([0:1], [0:1])$.
Indeed,
any central symmetric toric diagram corresponds to a balanced toric CY 3-fold.

\subsection{Equivariant elliptic genera of balanced toric Calabi-Yau 3-folds}

\begin{thm} \label{thm:Balanced}
Let $Y$ be a balanced toric CY 3-fold.
Then its equivariant elliptic genus is independent of $t_1$ and $t_2$,
and it is given by
\be
Z_Y(\tau, z;t_1,t_2) = \frac{\chi(Y)}{2} \frac{\theta_1(\tau, 2z)}{\theta_1(\tau, z)}.
\ee
\end{thm}

\begin{proof}
Let $\Gamma$ be  the toric diagram of $Y$.
For any vertex $v \in V(\Gamma)$,
let $v'$ be its balancing vertex.
Then by \eqref{eqn:Theta-Identity} we have
\ben
Z_Y(\tau, z; t_1, t_2)
& = & \half \sum_{v\in\Gamma}
\biggl( \prod_{j=1}^3 \frac{\theta_1(\tau, z+ w^v_j)}{\theta_1(\tau, w^v_j)}
+ \prod_{j=1}^3 \frac{\theta_1(\tau, z+ w^{v'}_j)}{\theta_1(\tau, w^{v'}_j)} \biggr) \\
& = & \half \sum_{v\in\Gamma}
\biggl( \prod_{j=1}^3 \frac{\theta_1(\tau, z+ w^v_j)}{\theta_1(\tau, w^v_j)}
+ \prod_{j=1}^3 \frac{\theta_1(\tau, z- w^{v}_j)}{\theta_1(\tau, -w^{v}_j)} \biggr) \\
& = & \half \sum_{v\in\Gamma}\frac{\theta_1(\tau, 2z)}{\theta_1(\tau, z)}\\
& = & \frac{\chi(Y)}{2} \frac{\theta_1(\tau, 2z)}{\theta_1(\tau, z)}.
\een
Here in the last equality we have used the fact that $\chi(Y)$ equals
the number of torus fixed points.
\end{proof}

\subsection{Averaged equivariant elliptic genera of toric CY 3-fold}

Inspired by the result of last section,
we introduce a notion of the averaged equivariant elliptic genus of
a toric CY 3-fold:
\be
Z^{av}_Y(\tau, z; t_1, t_2)
= \half (Z_Y(\tau, z; t_1, t_2) + Z_Y(\tau, z; -t_1, -t_2)).
\ee
By almost the same proof of Theorem \ref{thm:Balanced} one gets:

\begin{thm}
Let $Y$ be a toric 3-fold.
Then its averaged equivariant elliptic genus is independent of $t_1$ and $t_2$,
and it is given by
\be
Z^{av}_Y(\tau, z;t_1,t_2)
= \frac{\chi(Y)}{2} \frac{\theta_1(\tau, 2z)}{\theta_1(\tau, z)}.
\ee
\end{thm}

This establishes \cite[(3.48)]{Eguchi-Sugawara} for general toric CY 3-folds.
By comparing with Example 3 in Eguchi-Sugawara \cite{Eguchi-Sugawara},
our averaging of the equivariant elliptic genus seems to play the role
of their assumption of the ``charge conjugation symmetry".

\section{A Definition of Elliptic Genera of Noncompact Complex Manifolds}

In this section we propose a definition of elliptic genus of
a noncompact complex manifold
with a compactfication to a compact smooth manifold by adding a
divisor with normal crossing singularities.
It is inspired by the construction of mixed Hodge structures on
the cohomology nonsingular quasiprojective varieties by Deligne \cite{Deligne}.

\subsection{The case of compactification by a smooth divisor}

Suppose that $D \subset M$ is a smooth divisor of a compact complex manifold $M$,
and $U = M - D$.
Denote by $\nu_{D/M}$ the normal bundle of $D$ in $M$.
We define
\be
Z_U(\tau, z)
= \chi(M, \Ell(T^*M; \tau, z))
- \chi(D, S(\nu^*_{D/M}) \otimes \Ell(T^*M|_D)).
\ee
The reason for making this definition is that when there is a $T^2$-action on
a compact 3-fold $M$ such that $D$ is also $T^2$-invariant,
then by considering the equivariant version we can get back the equivariant elliptic
genus used in preceding  sections.

\subsection{The general case}

Suppose that $D = D_1 \cup \cdots \cup D_N  \subset M$
is a divisor of normal crossing singularities  of a compact complex manifold $M$,
and $U = M - D$.
Denote by $\nu_{D/M}$ the normal bundle of $D$ in $M$.
For $I =\{i_1, \dots, i_m\} \subset \{1, \dots, n\}$, $|I| =m$,
set
\be
D_I = D_{i_1} \cap \cdots \cap D_{i_m}.
\ee
We define
\be
\begin{split}
Z_U(\tau, z)
= & \chi(M, \Ell(T^*M; \tau, z)) \\
+ & \sum_{m=1}^N (-1)^m \sum_{|I|=m}
\chi(D_I, S(\nu^*_{D_I/M}) \otimes \Ell(T^*M|_{D_I})).
\end{split}
\ee
We expect that with this definition,
$Z_U$ is independent of the choice of the compatification,
and furthermore,
when $c_1(M) = [D_1] + \cdots + [D_N]$,
then $Z_U$ is a weak Jacobi form.

\end{document}